\documentclass[final,1p,times]{elsarticle}

\usepackage{amssymb}
\usepackage{amsthm}
\usepackage{amsmath,amssymb,amsopn,amsfonts,mathrsfs,amsbsy,amscd}
\usepackage{caption}
\usepackage{multirow}
\usepackage{longtable}
\usepackage{tikz}
\usetikzlibrary{matrix,decorations.pathreplacing, calc, positioning,fit}

\newcommand{\sln}{\mathrm{sl}(n,\mathbb{R})}

\newcommand{\prs}{\langle\;,\;\rangle}

\newcommand{\too}{\longrightarrow}
\newcommand{\om}{\omega}

\newcommand{\esp}{\quad\mbox{and}\quad}

\def\br{[\;,\;]}

\newcommand{\G}{\mathfrak{g}}
\newcommand{\g}{\mathfrak{g}}

\newcommand{\h}{{\mathfrak{h}}}

\newcommand{\ad}{{\mathrm{ad}}}

\newcommand{\tr}{{\mathrm{tr}}}

\newcommand{\spa}{{\mathrm{span}}}
\newcommand{\B}{{\cal B}}

\newcommand{\di}{\displaystyle}

\newcommand{\al}{\alpha}
\newcommand{\be}{\beta}

\newcommand{\Ga}{\Gamma}
\newcommand{\e}{\epsilon}

\newcommand{\ch}{\check{e}}
\newcommand{\la}{\lambda}

\newcommand{\de}{\delta}

\newtheorem{theo}{Theorem}[section]
\newtheorem{pr}{Proposition}[section]
\newtheorem{Le}{Lemma}[section]
\newtheorem{co}{Corollary}[section]

\newtheorem{exem}{Example}
\newtheorem{remark}{Remark}

\font\bb=msbm10

\def\B{\hbox{\bb B}}
\def\R{\hbox{\bb R}}

\def\N{\hbox{\bb N}}

\begin{document}

\begin{frontmatter}
	
	
	
	
	\title{ On the Existence and Properties of Left Invariant $k$-Symplectic Structures on Lie Groups with Bi-Invariant Pseudo-Riemannian Metric }
	
	\author[label1,label2]{Ait Brik Ilham, Mohamed Boucetta}
	\address[label1]{Universit\'e Hassan II\\ Facult\'e des Sciences Ain Chock\\e-mail: ilham.aitbrik@gmail.com 
		
	}
	\address[label2]{ Universit\'e Cadi-Ayyad\\
		Facult\'e des sciences et techniques\\
		BP 549 Marrakech Maroc\\e-mail: m.boucetta@uca.ac.ma
	}
	

	
	
	
	
	\begin{abstract} $k$-symplectic manifolds are a convenient framework to study  classical field theories and they are a generalization of polarized symplectic manifolds. This paper focus on  the existence and the properties of left invariant $k$-symplectic structures on Lie groups having a bi-invariant pseudo-Riemannian metric. 
		We show that compact semi-simple Lie groups and a large class of Lie groups  having a bi-invariant pseudo-Riemannian metric does not carry any left invariant $k$-symplectic structure. This class contains the oscillator Lie groups which are the only solvable non abelian Lie groups having a bi-invariant Lorentzian metric. However, we built a natural left invariant $n$-symplectic structure on $\mathrm{SL}(n,\R)$. Moreover, up to dimension 6, only three connected and simply connected Lie groups have a bi-invariant indecomposable pseudo-Riemannian metric and a left invariant k-symplectic structure, namely, the universal covering of $\mathrm{SL}(2, \R)$ with a 2-symplectic structure, the universal covering of the Lorentz group $\mathrm{SO}(3, 1)$ with a 2-symplectic structure, and a 2-step nilpotent 6-dimensional connected and simply connected Lie group with both a 1-symplectic structure and a 2-symplectic structure.

	\end{abstract}
	
	\begin{keyword} $k$-symplectic structure\sep semi-simple Lie algebras \sep  bi-invariant metric
		\MSC 17B60 \sep \MSC 17B99 
		
		
	\end{keyword}
	
\end{frontmatter}

\section{Introduction}\label{section1}

Symplectic geometry is the natural arena to develop classical mechanics.
The $k$-symplectic geometry is a generalization of the polarized symplectic geometry which was developed by  A. Awane \cite{awane2},   Awane and Goze \cite{Goze} and  C. G\"unther in \cite{gunther} as an attempt to develop a convenient geometric framework to study classical field theories (see \cite{Leon}). This geometry was also an attempt to formalize Nambu mechanic which  is a generalization of Hamiltonian mechanics involving multiple Hamiltonians.

A $k$-symplectic manifold is a smooth manifold $M$ of dimension $(k+1)n$ endowed with an involutive vector subbundle $E\subset TM$ and a family $(\om_1,\ldots,\om_k)$ of closed differential   2-forms such that:
$\mathrm{rank}(E)=nk$, the family $(\om_1,\ldots,\om_k)$ is nondegenerate,  i.e.,  $\di\cap_{i=1}^k\ker\om_i=\{0\}$ and $E$ is isotropic with respect all the $\om_i$. A 1-symplectic manifold is a symplectic manifold endowed with a Lagrangian foliation. There are few known examples of $k$-symplectic manifolds and one way to build new one is to consider  the class of  left invariant $k$-symplectic structures on a connected   Lie group $G$ and their quotients by lattices. In this case, a left invariant $k$-symplectic structure on a connected Lie group $G$  of dimension $(k+1)n$ is equivalent to the data of a Lie subalgebra $\h$ of dimension $nk$ of the Lie algebra $\G$ of $G$ and a family $\{\theta_i\in\wedge^2\G^*,i=1,\ldots,k\}$ of  nondegenerate 2-cocycle of $\G$  such that $\theta_i(\h,\h)=0$ for $i=1,\ldots,n$. We call $(\G,\h,\theta_1,\ldots,\theta_k)$ a $k$-symplectic Lie algebra. The determination of Lie groups having a left invariant $k$-symplectic structure is an open problem and their study is at its debut \cite{Awane, abchir}.

In this paper, we study left invariant $k$-symplectic structures on Lie groups having a bi-invariant pseudo-Riemannian metric. This class of Lie group is very large and contains all semi-simple Lie groups. They became relevant some years ago when they were useful in the formulation of some physical problems, for instance in the so known Adler-Kostant-Symes scheme. More recently they appeared in conformal field theory \cite{fig}. One can see \cite{ovando}, for a survey on the Lie algebras of Lie groups having a bi-invariant metric.

 In Section \ref{section3},
we show  that compact semi-simple Lie groups and a large class of Lie groups  having a bi-invariant pseudo-Riemannian metric does not carry any left invariant $k$-symplectic structure. This class contains the oscillator Lie groups which are the only solvable non abelian Lie groups having a bi-invariant Lorentzian metric. We end the section, by   building a natural left invariant $n$-symplectic structure on $\mathrm{SL}(n,\R)$. Moreover, In Section \ref{section5}, we show that, up to dimension 6, only three connected and simply connected Lie groups have a bi-invariant indecomposable pseudo-Riemannian metric and a left invariant k-symplectic structure, namely, the universal covering of $\mathrm{SL}(2, \R)$ with a 2-symplectic structure, the universal covering of the Lorentz group $\mathrm{SO}(3, 1)$ with a 2-symplectic structure, and a 2-step nilpotent 6-dimensional connected and simply connected Lie group with both a 1-symplectic structure and a 2-symplectic structure.

\section{Left invariant $k$-symplectic structures on quadratic Lie groups}\label{section3}

A $k$-symplectic manifold is a smooth manifold $M$ of dimension $(k+1)n$ endowed with a vector subbundle $E\subset TM$ and a family $(\om_1,\ldots,\om_k)$ of differential   2-forms such that:\begin{enumerate}
	\item[$(i)$] $E$ is of rank $nk$ and is involutive, i.e, for any $X,Y\in\Ga(E)$, $[X,Y]\in\Ga(E)$ hence $E$ defines a foliation on $M$,
	 \item[$(ii)$] for each $i\in\{1,\ldots,k\}$, $d\om_i=0$ and the family $(\om_1,\ldots,\om_k)$ is nondegenerate,  i.e.,  $\di\cap_{i=1}^k\ker\om_i=\{0\}$,  \item[$(iv)$] $E$ is isotropic with respect all the $\om_i$, i.e., $\om_i(E,E)=0$.
\end{enumerate}

A 1-symplectic manifold is a symplectic manifold endowed with a Lagrangian foliation.

A left invariant $k$-symplectic structure on a Lie group $G$ is a $k$-symplectic structure $(E,\om_1,\ldots,\om_k)$ on $G$ such that $E$ and the $\om_i$ are invariant by all the left multiplications of $G$. In particular, $\h=E_e$ is a Lie subalgebra of the Lie algebra $\G$ of $G$ and $\om_1(e),\ldots,\om_k(e)$ are 2-cocycles of $\G$. 

A $k$-symplectic Lie algebra is a real Lie algebra $\G$ of dimension $nk+n$ with a	subalgebra $\h$ of dimension $nk$ and a family $(\theta_1,\ldots,\theta_k)$ of  2-forms satisfying:
\begin{enumerate}\item[$(i)$] The family $(\theta_1,\ldots,\theta_k)$ is nondegenerate, i.e., $\bigcap_{i=1}^k\ker\theta_i=\{0\}$,
	\item[$(ii)$] for $i=1,\ldots,k$, $\theta_i$ is a 2-cocycle, i.e., $\mathrm{d}\theta_i(u,v,w):=\theta_i([u,v],w)+\theta_i([v,w],u)+\theta_i([w,u],v)=0$,
	\item[$(iii)$] $\h$ is totally isotropic with respect to $(\theta_1,\ldots,\theta_k)$, i.e., $\theta_i(u,v)=0$ for any $u,v\in\h$ and for any $i=1,\ldots,k$.
	
\end{enumerate}
There is a correspondence between left invariant $k$-symplectic structures on a Lie group and $k$-symplectic structures on its Lie algebra.

Note that a symplectic Lie algebra is a Lie algebra endowed with a nondegenerate 2-cocycle and a 1-symplectic Lie algebra is a symplectic Lie algebra $(\G,\theta)$ having a Lagrangian subalgebra $\h$, i.e., $\dim\h=\frac12\dim\G$ and $\theta(\h,\h)=0$.

\begin{exem} \begin{enumerate}
		\item Let $\G$ be an abelian Lie algebra of dimension $n(k+1)$, $(e_{pi},e_i)_{1\leq p\leq k,1\leq i\leq n}$ a basis of $\G$ and $(\om^{pi},\om^i)_{1\leq p\leq k,1\leq i\leq n}$ its dual basis. For  any $\al\in\{1,\ldots,k\}$, put 
		\[ \theta_\al=\sum_{i=1}^n\om^{\al i}\wedge\om^i\esp \h=\ker\om^1\cap\ldots\cap\ker\om^n. \]
		Then $(\h,\theta_1,\ldots,\theta_k)$ is a $k$-symplectic structure of $\G$. Thus any abelian Lie algebra of dimension $n(k+1)$ has a $k$-symplectic structure.
		
		\item Let $\G=\mathrm{span}\{e_1,e_2,f_1,f_2,f_3,f_4  \}$ be the 6-dimensional Lie algebra where the non vanishing Lie brackets are given by
		\[\begin{cases}
			[f_{1},f_{2}]=-f_{1},\ [f_{1},f_{4}]=-a f_{1}, \ [f_{2},f_{3}]=f_{3}, \ [f_{3},f_{4}]=-a f_{3}, \\ 
			[f_{2},e_{1}]=-e_{1},\ [f_{2},e_{2}]=-c(af_2-f_4), \ [f_{4},e_{1}]= -a e_{1}, \ [f_{4},e_{2}]=-b(af_2-f_4),\quad a,b,c\in\R \end{cases}\]The triple
		\[ \h=\mathrm{span}\{f_1,f_2,f_3,f_4  \}, \;\theta_1=f_1^*\wedge e_1^*+f_2^*\wedge e_2^*\esp \theta_2=f_3^*\wedge e_1^*+f_4^*\wedge e_2^* \]defines a 2-symplectic structure on $\G$.
		 
	\end{enumerate}

\end{exem}

Let $G$ be a Lie group and $(\G=T_eG,\br)$ its Lie algebra. A pseudo-Riemannian metric $g$ on  $G$ is called bi-invariant if its invariant by the left and the right multiplications of $G$. This is equivalent to $\prs=h(e)$ satisfying
\begin{equation}\label{quadratic} \langle [u,v],w\rangle+\langle [u,w],v\rangle=0, 
\end{equation}
for any $u,v,w\in\G$. This means that, for any $u\in\G$,
 $\ad_u$ is skew-symmetric with respect to $\prs$ where $\ad_u(v)=[u,v]$. 
 
 A Lie group endowed with a bi-invariant metric will be called {\it quadratic} and a real Lie algebra endowed with a bilinear symmetric nondegenerate form satisfying \eqref{quadratic}  will be called a {\it quadratic Lie algebra}. A metric $\prs$ satisfying \eqref{quadratic} will be called quadratic or invariant. There is a correspondence between connected quadratic Lie groups and quadratic Lie algebras so we will work at the level of Lie algebras.

There are three important classes of quadratic Lie algebra: semi-simple Lie algebras endowed with their Killing form, $T^*$-extensions of Lie algebras and quadratic Lie algebras obtained by the process of double extension (see \cite{ovando} for more details).

\begin{enumerate}
\item[$(a)$]
Recall that a semi-simple Lie algebra is a Lie algebra $\G$ such that its Killing form given by $k(u,v)=\tr(\ad_u\circ\ad_v)$ is nondegenerate. Since $k$ is invariant, $(\G,k)$ is a quadratic Lie algebra.
\item[$(b)$]
The $T^*$-extension of a Lie algebra $\G$ is the Lie algebra $T^*\G=\G\oplus\G^*$ where the Lie bracket is given by
\begin{equation}\label{br} [u+\al,v+\be]=[u,v]+\ad_u^*\be-\ad_v^*\al,\quad u,v\in\G,\al,\be\in\G^*, 
\end{equation}
where $\ad_u^*\al(v)=-\al([u,v])$.
The bilinear symmetric form $\prs$ on $T^*\G$ given by
\begin{equation}\label{met} \langle u+\al,v+\be\rangle= \al(v)+\be(u) 
	\end{equation}
is nondegenerate and invariant and hence $(T^*\G,\prs)$ is a quadratic Lie algebra.

\item[$(c)$]
Let $(\G,\br_\G,\prs_\G)$ be a quadratic  Lie algebra. The double extension of $\G$ by the mean of $A\in\mathrm{so}(\G,\prs_\G)$ is the Lie algebra $\de_A(\G)=\R e\oplus\G\oplus\R \bar{e}$ where the non vanishing Lie brackets and the metric are given, for any $u,v\in\G$, by
\[ [\bar{e},u]=Au,\; [u,v]=\langle Au,v\rangle_\G e+[u,v]_\G,\; \langle xe+u+\bar{x}\bar{e},xe+u+\bar{x}\bar{e}\rangle =2x\bar{x}+\langle u,u\rangle_\G.  \]
$(\de_A(\G),\br,\prs)$ is a quadratic Lie algebra. 
\end{enumerate}

Let us start by a well-known result.
\begin{pr}\label{nosym}
	Let $(\G,\br,\prs)$ be a quadratic Lie algebra. If $\G$ has a symplectic form $\theta$ then $\G$ is nilpotent.
\end{pr}
\begin{proof} It a consequence of the fact that the endomorphism $D$ given by $\theta(u,v)=\langle Du,v\rangle$ is an invertible derivation and, by virtue of Jacobson's Theorem \cite{jacobson}, a Lie algebra with an invertible derivation must be nilpotent.
\end{proof}

However, non nilpotent quadratic Lie algebras can carry a $k$-symplectic structure with $k\geq2$ as the following example shows.

\begin{exem}\label{exem2}
	We consider $\mathrm{sl}(2,\R)$ with its   basis
	$\left\{ e_1=\left(\begin{matrix}0&1\\0&0\end{matrix} \right), e_2=\left(\begin{matrix}0&0\\1&0\end{matrix} \right),e_3=\left(\begin{matrix}1&0\\0&-1\end{matrix} \right)\right\},$ where
	\begin{equation*} \label{eqsl2} [e_3,e_1]=2e_1,\;[e_3,e_2]=-2e_2\esp [e_1,e_2]=e_3. \end{equation*}
	Then $\h_0=\mathrm{span}\{h,g\}$, $\theta_1=e_3^*\wedge e_2^*+b e_1^*\wedge e_2^*$ and $\theta_2=e_1^*\wedge e_2^*$ is a 2-symplectic structure on $\mathrm{sl}(2,\R)$.
\end{exem}

Let $(\G,\h,\theta_1,\ldots,\theta_k)$ be a $k$-symplectic Lie algebra of dimension $n(k+1)$.
The linear map $\Theta:\h\too (\G/\h)^*\times\ldots\times (\G/\h)^*$, $h\mapsto (\Theta_1(h),\ldots,\Theta_k(h))$ where, for any $p\in\G$,
\[ \Theta_\al(h)([p])=\theta_\al(h,p) \]
is well-defined, injective and for dimensional reasons it is an isomorphism. For any $\al\in\{1,\ldots,k\}$, the vector subspace $\h^\al$  of $\h$ given by
\[ \h^\al=\{h\in\h,\Theta_\be(h)=0,\be=1,\ldots,k,\be\not=\al \}. \]
has dimension $n$ and $\di\h=\oplus_{\al=1}^k\h^\al$.

Suppose now that $\G$ carries an invariant metric $\prs$. Then, for any $\al=1,\ldots,k$, there exists a skew-symmetric endomorphism $D_\al:\G\too\G$ such that, for any $u,v\in\G$,
\[ \theta_\al(u,v)=\langle D_\al u,v\rangle. \]
The fact that $\theta_\al$ is closed is equivalent to $D_\al$ is a derivation of $(\G,\br)$ and the fact that $(\theta_1,\ldots,\theta_k)$ is nondegenerate is equivalent to $\bigcap_{\al}\ker D_\al=\{0\}$. Moreover, $\theta_\al(\h,\h)=0$ is equivalent to $D_\al(\h)\subset\h^\perp$, where $\h^\perp$ is the orthogonal of $\h$ with respect to $\prs$.

In conclusion, a $k$-symplectic structure on a quadratic Lie algebra $(\G,\prs)$ of dimension $n(k+1)$ is given by a subalgebra $\h$ of dimension $nk$ and a family of skew-symmetric derivations $(D_1,\ldots,D_k)$ such that 
$\bigcap_{\al}\ker D_\al=\{0\}$ and  $D_\al(\h)\subset\h^\perp$ for any $\al\in\{1,\ldots,k\}$.

Note that in a quadratic Lie algebra $(\G,\prs)$, \begin{equation}\label{ad}H^2(\G)\simeq \mathrm{Der}(\G)\cap\mathrm{so}(\G,\prs)/\{\ad_u,u\in\G  \},
\end{equation}
where $\mathrm{Der}(\G)$ is the Lie algebra of derivations of $\G$ and $\mathrm{so}(\G,\prs)$ is the Lie algebra of skew-symmetric endomorphisms.
\begin{pr}\label{center} Let $(\G,\br,\prs)$ be a quadratic Lie algebra. Then the following assertions hold.

\begin{enumerate}
	\item If $H^2(\G)=\{0\}$ and the center $Z(\G)\not=\{0\}$ then $G$ has no $k$-symplectic structure.
	\item If $\dim Z(\G)=1$ and $Z(\G)$ is nondegenerate then, for any skew-symmetric derivation $D$, $D(Z(\G))=0$. In particular, $\G$ has no $k$-symplectic structure.
	
\end{enumerate}

\end{pr}

\begin{proof}   \begin{enumerate}
		\item If $H^2(\G)=\{0\}$ then any slew-symmetric derivation is inner and hence trivial on the center.
		\item Put $Z(\G)=\R e_0$ and let $D$ be a skew-symmetric derivation. Since the derivations preserve the center, then $D(e_0)=\al e_0$. If $\langle e_0,e_0\rangle \not=0$ then $\al=0$.\qedhere
	\end{enumerate}
\end{proof}

\begin{remark} \label{rem}  Note that a quadratic Lie algebra $(\G,\prs)$ satisfies
	$[\G,\G]^\perp=Z(\G)$ and hence a solvable quadratic Lie algebra satisfies $Z(\G)\not=\{0\}$.

\end{remark}

\begin{pr}\label{nilpotent}Let $(\G,\br,\prs)$ be a quadratic nilpotent Lie algebra of dimension $n$. If $\G$ admits a $n-1$-symplectic structure then $\G$ is abelian.
	
\end{pr}
\begin{proof} Suppose that $\G$ carries a $n-1$-symplectic structure. Then there exists a $n-1$-dimensional Lie subalgebra $\h$, a family $D_1,\ldots,D_{n-1}$ of skew-symmetric derivations such that $D_i(\h)\subset\h^\perp$ and $\bigcap_{i}\ker D_i=\{0\}$. Since $\h$ is a subalgebra and $\prs$ is quadratic then $[\h,\h^\perp]\subset\h^\perp$. But $\dim\h^\perp=1$ and $\ad_u$ is nilpotent for any $u$ and hence $[\h,\h^\perp]=\{0\}$. Now the facts that $D_i(\h)\subset\h^\perp$, $[\h,\h^\perp]=\{0\}$ and $D_i$ a derivation implies that $D_i([\h,\h])=\{0\}$ for $i=1,\ldots,n-1$ and hence $[\h,\h]=\{0\}$. Put $\G=\h\oplus\R e$. We have $[\g,\g]=[\h,\h]+[\h,e]=[\h,e]$ and for any $x,y\in\h$,
	\[ \langle [e,x],y\rangle=\langle e,[x,y]\rangle=0. \]So $[\G,\G]\subset\h^\perp$ and hence $\h\subset[\G,\G]^\perp=Z(\G)$. This implies that $\G$ is abelian which completes the proof.
\end{proof}

Let us state our first main result.

\begin{theo}\label{main} Let $(\G,\prs)$ be a quadratic  Lie algebra  and $(\h,\theta_1,\ldots,\theta_k)$  a $k$-symplectic structure on $\G$. If $\h$ is nondegenerate with respect to $\prs$ then $([\theta_1],\ldots,[\theta_k])$ are linearly independent in $H^2(\G)$. In particular, if $k>\dim H^2(\G)$ then $\h$ is degenerate.
\end{theo}	

\begin{proof} For  $\al=1,\ldots,k$,
	\[ \theta_\al(u,v)=\langle D_\al u,v\rangle \]where $D_\al$ is a skew-symmetric derivation of $\G$.
	Put $\h^\al=\bigcap_{\be\not=\al}(\ker D_\al\cap\h)$. Then
	\[ \h=\h^1\oplus\ldots\oplus\h^k \]and $(D_\al)_{|\h^\al}:\h^\al\too \h^\perp$ is an isomorphism. Suppose that $\h$ is nondegenerate and $([\theta_1],\ldots,[\theta_k])$ are linearly dependent in $H^2(\G)$. Then there exists $(a_1,\ldots,a_k)\not=(0,\ldots,0)$ and $x\in\G$ such that 
	\[ \ad_x=a_1D_1+\ldots+a_kD_k. \]
	Suppose that $a_1\not=0$ and put $x=z+t$ where $z\in\h$ and $t\in\h^\perp$.  Since $\h$ is a subalgebra and the metric is invariant then, for any $h\in\h^1$, $[h,\h^\perp]\subset\h^\perp$ and hence
	\[ [x,h]=a_1D_1(h)=[z,h]+[t,h]=[t,h]. \]
	Then $t\not=0$ and $(D_1)_{|\h^1}=\frac{1}{a_1}(\ad_t)_{|\h^1}$. But
	$\langle [t,h],t\rangle=0$ so $(D_1)_{|\h^1}(\h^1)\subset t^\perp\cap\h^\perp$. But this is impossible since $$\dim (t^\perp\cap\h^\perp)=\dim\G-1-\dim\h=\dim\h^\perp-1. $$This completes the proof.
\end{proof}

This theorem has two important corollaries.

\begin{co}\label{co1} Let $(\G,\prs)$ be a quadratic Lie  Lie algebra with $H^2(\g)=\{0\}$ and $(\h,\theta_1,\ldots,\theta_k)$ a $k$-symplectic structure on $\G$.  Then $\h$ is degenerate with respect to $\prs$.
\end{co}
Recall that a the connected and simply connected Lie group associated to a semi-simple Lie algebra is compact if and only if  the Killing form is definite negative.
\begin{co}\label{co2}
	Let $\G$ be a semi-simple Lie algebra and $(\h,\theta_1,\ldots,\theta_k)$ a $k$-symplectic structure on $\G$. Then $\h$ is degenerate with respect to the Killing form of $\G$. In particular, a compact semi-simple Lie group has no left invariant $k$-symplectic structure.
\end{co}

The following theorem shows that, in addition to  the class of semi-simple Lie algebras, there is a large class of quadratic Lie algebras satisfying $H^2(\G)=\{0\}$.

\begin{theo}\label{h2} If $\G$ is a simple Lie algebra then $H^2(T^*\G)=0$.
	
\end{theo}

\begin{proof} Recall that the quadratic structure of $T^*\G=\G\oplus\G^*$ is given by \eqref{br} and \eqref{met}.
	For $\om\in\wedge^2\G^*$ and $\pi\in\wedge^2\G$ we denote $\om^\flat:\G\too\G^*$ and $\pi_\#:\G^*\too\G$ given by $\prec \om(x),y\succ=\om(x,y)$ and $\prec\be,\pi_\#(\al)\succ=\pi(\al,\be)$. 
	According to \eqref{ad}, we must show that any skew-symmetric derivation of $T^*\G$ is inner. Note first that for any $x+\al\in T^*\G$, we have
	\[ \ad_{x+\al}=\left(\begin{matrix} \ad_x&0\\ -(d\al)^\flat&\ad_x^* \end{matrix}  \right), \]where $d\al(u,v)=-\al([u,v])$ and, for any skew-symmetric derivation $D$, there exists $\om\in\wedge^2\G^*$,  $\pi\in\wedge^2\G$ and $D_1:\G\too\G$ such that
	\[ D=\left(\begin{matrix} D_1&\pi_\#\\\om^\flat&-D_1^* \end{matrix} \right), \]where $D_1^*$ is the dual of $D_1$. One can see easily that $D$ is a derivation if and only if $D_1$ is a derivation of $\G$, $\om$ is a 2-cocycle of $\G$ and $\pi$ is $\ad^*$-invariant, i.e., for any $x\in\G$ and $\al,\be\in\G^*$,
	\[ \pi(\ad_x^*\al,\be)+\pi(\al,\ad_x^*\be)=0. \]
	Since $\G$ is simple then $\om=d\al$ and $D_1=\ad_x$. On the other hand, let $A=\pi_\#\circ\kappa^\flat$.  Since both $\pi$ and $\kappa$ are $\ad$-invariant then $[A,\ad_x]=0$ for any $x\in\G$. The complexification $A^{\mathbb{C}}$ of $A$ is an endomorphism of $\G^{\mathbb{C}}$ which commutes with $\ad_x$ for any $x\in \G^{\mathbb{C}}$ and any eigenspace of $A^{\mathbb{C}}$ is an ideal which implies that $A=\la \mathrm{Id}_\G$ and since $A$ is skew-symmetric then $\la=0$ which completes the proof.
	\end{proof}

\begin{theo}\label{de} Let $\de_A(\G)=\R e\oplus\G\oplus\bar{e}$ be the double extension of an abelian quadratic Lie algebra $(\G,\prs_\G)$ with $A$ non nilpotent and  let $(\h,\prs_\h)$ be an abelian quadratic Lie algebra. Then, for any skew-symmetric derivation $D$ of $\de_A(\G)\oplus\h$ (product of the quadratic Lie algebras $\de_A(\G)$ and $\h$), $D(e)=0$. In particular, $\de_A(\G)\oplus\h$ has no $k$-symplectic structure.
	
\end{theo}

\begin{proof}Let $D$ be a skew-symmetric derivation of  $\de_A(\G)\oplus\h$.  We have $Z(\de_A(\G)\oplus\h)=\R e\oplus \ker A\oplus \h $. Then $D(e)=\al e+z$ where $z\in\ker A\oplus\h$. For any $u\in\G$, 
	$$Du=\langle Du,\bar{e}\rangle e+D_1u+\langle Du,e\rangle\bar{e}+D_2u=\langle Du,\bar{e}\rangle e+D_1u+D_2u,$$ where $D_1u\in\G$ 	and $D_2u\in\h$. Note that $D_1:\G\too\G$ is skew-symmetric with respect to $\prs_\G$. We have
	\begin{align*}
		D[\bar{e},u]&=DAu=\langle DAu,\bar{e}\rangle e+D_1Au+D_2Au\\
		&=[D\bar{e},u]+[\bar{e},Du]\\
		&=\langle D\bar{e},e\rangle [\bar{e},u]+[D_1\bar{e},u]+AD_1u\\
		&=-\al Au+\langle D_1\bar{e},u\rangle_\G e+AD_1u.
	\end{align*}So
	$ \al A=[A,D_1].$
	On the other hand, for any $u,v\in\G$,
	\begin{align*}
		D[u,v]&=\langle Au,v\rangle_\G D(e)\\
		&=[Du,v]+[u,Dv]\\
		&=\left(\langle AD_1u,v\rangle_\G+\langle u,AD_1v\rangle_\G\right) e\\
		&=\langle[A,D_1]u,v\rangle e\\
		&=\al \langle Au,v\rangle e.
	\end{align*}So $D(e)=\al e$. If $A$ is not nilpotent, then there exists $n\geq2$ such that $\tr(A^n)\not=0$. But from the relation $\al A=[A,D_1]$ we can deduce that $\al\tr(A^n)=0$ and hence   
	$D(e)=0$ which completes the proof.
\end{proof}

We  apply Theorem \ref{de} to an important class of quadratic Lie algebras, namely, the oscillator Lie algebras.

The oscillator group named so by Streater in \cite{streater}, is a four-dimensional connected, simply connected
Lie group, whose Lie algebra (known as the oscillator algebra) coincides with the one generated
by the differential operators, acting on functions of one variable, associated to the harmonic
oscillator problem. The oscillator group has been generalized to any even dimension $2n \geq 4$, there are the only solvable Lie groups which carry an invariant Lorentzian metric.

For $n \in \N^*$ and $\la = (\la_1,\ldots,\la_n) \in \R^n$ with $0 < \la_1 \leq \cdots \leq \la_n$, the $\la$-oscillator group, denoted by $\mathrm{0sc}_\la$, is the Lie group which the underlying manifold $\R^{2n+2} = \R \times \R\times \mathbb{C}^n$ and product $$(t,s,z).(t',s',z') = \left(t+t',s+s'+\frac12 \sum_{j=1}^n Im\bar{z}_j exp(it\la_j)z_j',
\ldots, z_j+exp(it\la_j)z'_j, \ldots \right).$$

The Lie algebra of $\mathrm{0sc}_\la$, denoted by $\mathrm{osc}(\la)$, admits a basis $\B = \left\{e_{-1},e_0,e_i,\ch_i\right\}_{i=1,\ldots,n}$ where the brackets are given by
\begin{equation*}\label{bracket}
	[e_{-1},e_i] = \la_i \ch_i, \qquad[e_{-1},\ch_i] = -\la_ie_i, \qquad[e_i,\ch_i] = e_0,
\end{equation*}
the unspecified brackets are either zero or given by antisymmetry. 
We note that the center of $\mathrm{osc}(\la)$ is ${Z}(\mathrm{osc}(\la)) = \R e_0$ and $[\mathrm{osc}(\la), \mathrm{osc}(\la)] = \mathrm{span}_{\R}\left\{e_0,e_i,\ch_i\right\}_{i=1,\ldots,n}.$

For $x=x_{-1}e_{-1}+x_0e_0+\sum_{i=1}^n(x_i e_i+y_i \ch_i)$, the quadratic metric on $\mathrm{osc}(\la)$ is given by
\[ \langle x,x\rangle=2x_{-1}x_0+\sum_{i=1}^n\frac1{\la_i}(x_i^2+y_i^2). \]
Hence $\mathrm{osc}(\la)$ is the double extension of  $\R^{2{n}}$ by mean of $A=\mathrm{Diag}(A_{\al_1},\ldots,A_{\al_n})$ where $A_{\la_i}=\left(\begin{matrix}
	0&-\la_i\\ \la_i&0
\end{matrix}  \right).$ So as a corollary of Theorem \ref{de} we get the following result.

\begin{co}\label{osc}Let $(\h,\prs_\h)$ be an abelian quadratic Lie algebra and $\mathrm{osc}(\la)\times\h$ endowed with the  quadratic metric product. If $D$ is a  skew-symmetric derivation of $\mathrm{osc}(\la)\times\h$ then $D(e_{0})=0$. In particular, there is no $k$-symplectic structure on $\mathrm{osc}(\la)\times\h$.
	
\end{co}

We end this  section, by building a  $n$-symplectic structure on $\mathrm{sl}(n,\R)$. 

For any $i,j$, we denote by $E_{i,j}$ the $n$-matrix with 1 in the $i$-row and the $j$-column and $0$ elsewhere and we denote by $(E_{i,j}^*)$ the dual basis of the basis $(E_{i,j})$ of $\mathrm{gl}(n,\R)$.
We define $(\h,\theta_1,\ldots,\theta_n)$ by 
\[ \h=\left\{A\in\sln, \exists \la\in\R\;\mbox{such that}\; A(e_n)=\la e_n       \right\}\esp\theta_\al=(dE_{\al,n}^*)_{|\mathrm{sl}(n,\R)}.\]
\begin{theo}\label{sl} $(\mathrm{sl}(n,\R),\h,\theta_1,\ldots,\theta_n)$ is a $n$-symplectic Lie algebra.
	
\end{theo}

\begin{proof} We denote by $(e_1,\ldots,e_n)$ the canonical basis of $\R^n$ and $(e_1^*,\ldots,e_n^*)$ its dual basis. We have $\dim\mathrm{sl}(n,\R)=(n-1)(n+1)$,  $\h$ is obviously a subalgebra of dimension $n(n-1)$ and the $\theta_i$ are obviously closed. Moreover, for any $A,B\in\h$, $[A,B](e_n)=0$ and hence, for any $\al=1,\ldots,n$, $$\theta_i(A,B)=-E_{\al,n}^*([A,B])=-
	\prec e_\al^*,[A,B](e_n)\succ=0.$$
	
Let $X\in\sln$ such that $i_X\theta_\al=0$ for any $\al=1,\ldots,n$. This equivalent to 
$$\prec e_\al^*,[X,E_{i,j}](e_n)\succ=0\quad i,j=1,\ldots,n,\al=1,\ldots,n.$$
We have
\begin{align*}
	\prec e_\al^*,[X,E_{i,j}](e_n)\succ&=
	\prec e_\al^*,\de_{jn}X(e_i)-X_{jn}e_i\succ\\
	&=\de_{jn}X_{\al i}-X_{jn}\de_{\al i}.
\end{align*}
For $j=n$ and $\al\not=i$, we get $X_{\al i}=0$. Otherwise, we get
\[ X_{\al\al}-X_{nn}=0. \]
So $X=X_{nn} I_n$. But $\tr(X)=0$ implies $X=0$. This completes the proof.\end{proof}

\section{k-symplectic structures on indecomposable quadratic Lie algebras of dimension $\leq6$}\label{section5}
In this section, we enumerate all indecomposable   quadratic Lie algebras of dimension $\leq6$ and for each of them we investigate if yes or not it admits a $k$-symplectic structure and we give such structure when it exists. Indecomposable quadratic Lie algebras of dimension $\leq6$, up to an automorphism, are given in \cite{ovando} and \cite{Baum} (see Tables \ref{1}-\ref{4}).

\paragraph{Dimension 3.} There are two quadratic non abelian Lie algebras in dimension 3, $\mathfrak{su}(2)$ which has no $k$-symplectic structure by virtue of Corollary \ref{co2} and $\mathrm{sl}(2,\R)$ which has a 2-symplectic structure given in Example \ref{exem2} and  Theorem \ref{sl}.

\paragraph{Dimension 4.}  In dimension 4, there are two non abelian quadratic Lie algebras, the oscillator Lie algebra $\mathrm{osc}(4,\la)$ and $\G_{1,4}$. Both are obtained by double extension from an abelian Lie algebra by mean of an invertible endomorphism and hence both of them have no $k$-symplectic structure by virtue of Theorem \ref{de}.

Note that , $\mathfrak{su}(2)\times\R$ and $\mathrm{sl}(2,\R)\times\R$ are the only decomposable non abelian four dimensional quadratic Lie algebras and both of theme have no $k$-symplectic structure by virtue of Proposition \ref{center}.

\paragraph{Dimension 5.} In dimension 5, there is only one indecomposable  non abelian  quadratic Lie algebra, namely, $\G_{1,5}$. Since $\G_{1,5}$ is nilpotent, it has no 5-symplectic structure by virtue of Proposition \ref{nilpotent}.

Note that the 5-dimensional decomposable quadratic Lie algebras $\mathrm{osc}(4,\la)\times\R$ and $\G_{1,4}\times\R$ have no $k$-symplectic structure by virtue of Theorem \ref{de}.

\paragraph{Dimension 6.} The list of indecomposable quadratic Lie algebras of dimension 6 is given in Table \ref{4}. The Lie algebras $\mathrm{osc}(6,\la_1,\la_2)$, $\mathfrak{l}_{2,\la}$ and $\mathfrak{n}_k(2,2)$ for $k=2,\ldots,6$ are obtained by
double extension from an abelian Lie algebra by mean of an invertible endomorphism and hence all of them have no $k$-symplectic structure by virtue of Theorem \ref{de}.

However, $\mathfrak{n}_1(2,2)$ is obtained by double extension by the mean of a nilpotent endomorphism. It is 2-step nilpotent and doesn't carry a 5-symplectic structure by virtue of Proposition \ref{nilpotent}. However, it carries a 2-symplectic structure.  For instance if we take $\h=\spa(e_1,e_2,e_4,e_3+e_6)$, $(\h,\theta_1,\theta_2)$ is a 2-symplectic structure on $\mathfrak{n}_1(2,2)$ where
$$\begin{cases}\theta_1=e_1^*\wedge e_3^*-e_1^*\wedge e_5^*-e_1^*\wedge e_6^*+e_3^*\wedge e_4^*+e_3^*\wedge e_6^*+e_4^*\wedge e_6^*,\\
	\theta_2=e_1^*\wedge e_3^*-e_1^*\wedge e_5^*-e_1^*\wedge e_6^*+2e_2^*\wedge e_5^*-e_3^*\wedge e_4^*+e_3^*\wedge e_6^*-e_4^*\wedge e_6^*.\end{cases}$$
Moreover, $(Z(\mathfrak{n}_1(2,2)),\theta)$ is a 1-symplectic structure on $\mathfrak{n}_1(2,2)$ where
\[ \theta=e_1^*\wedge e_3^* -e_1^*\wedge e_6^*+e_2^*\wedge e_5^*+e_2^*\wedge e_6^*+e_4^*\wedge e_5^*.\]

On the other hand, $\mathrm{so}(3,1)$ has no 1-symplectic structure by virtue of Proposition \ref{nosym}. However,
$(\h=\mathrm{span}(e_3,e_4,e_1+e_5,e_2+e_6),\theta_1,\theta_2)$ where 
$$\begin{cases}\theta_1=-e_1^*\wedge e_2^*-e_1^*\wedge e_4^*+e_2^*\wedge e_3^*+e_3^*\wedge e_6^*-e_4^*\wedge e_5^*+e_5^*\wedge e_6^*,\\
\theta_2=e_1^*\wedge e_3^*-e_1^*\wedge e_5^*+e_2^*\wedge e_4^*-e_2^*\wedge e_6^*+e_3^*\wedge e_5^*+e_4^*\wedge e_6^*\end{cases}$$
is a 2-symplectic structure on $\mathrm{so}(3,1)$. However, $\mathrm{so}(3,1)$ does not possess any five-dimensional Lie subalgebras \cite{ganam} and hence has no 5-symplectic structure.

Finally, we will show that $T^*\mathrm{sl}(2,\R)$ and 
$T^*\mathfrak{su}(2)$ have no $k$-symplectic structure. This needs some work based on the following two lemmas. But first let's recall the definitions of $\mathfrak{su}(2)$ and $\mathrm{sl}(2,\R)$ and set some notations.

\begin{enumerate}

\item The Lie algebra $\mathfrak{su}(2)=\left\{\left(\begin{matrix} iz&y+ix\\-y+xi&-zi\end{matrix}\right),x,y,z\in\R   \right\}.$ It has a basis $\B_0=(e_1,e_2,e_3)$
\[ e_1=\frac12\left(\begin{matrix}0&i\\i&0\end{matrix} \right),\; e_2=\frac12\left(\begin{matrix}0&1\\-1&0\end{matrix} \right)\esp e_3=\frac12\left(\begin{matrix}-i&0\\0&i\end{matrix} \right) \]
 where   
\begin{equation*} \label{eqsu2}[e_1,e_2]=e_3,\;[e_2,e_3]=e_1\esp [e_3,e_1]=e_2. 
\end{equation*} 
The group of automorphisms of $\mathfrak{su}(2)$ is generated by the three rotations
\[ \mathrm{Rot}_{xy}=\left(\begin{array}{ccc}\cos(a)&\sin(a)&0\\-\sin(a)& \cos(a)&0\\0&0&1 \end{array}   \right),
\mathrm{Rot}_{xz}=\left(\begin{array}{ccc}\cos(a)&0&\sin(a)\\0&1 &0\\
	-\sin(a)&0&\cos(a) \end{array}   \right),\mathrm{Rot}_{yz}=\left(\begin{array}{ccc}1&0&0\\0&\cos(a) &\sin(a)\\
	0&-\sin(a)&\cos(a) \end{array}   \right). \]

\item 
The Lie algebra $\mathrm{sl}(2,\R)$ has a   basis
$\left\{ e_1=\left(\begin{matrix}0&1\\0&0\end{matrix} \right), e_2=\left(\begin{matrix}0&0\\1&0\end{matrix} \right),e_3=\left(\begin{matrix}1&0\\0&-1\end{matrix} \right)\right\},$ where
\begin{equation*}  [e_3,e_1]=2e_1,\;[e_3,e_2]=-2e_2\esp [e_1,e_2]=e_3. \end{equation*}
In the basis $(X_1,X_2,X_3)$ with $X_1=\frac12(e_1+e_2)$, $X_2=\frac12e_3$ and $X_3=\frac12(e_1-e_2)$ the Lie brackets are given by
\begin{equation*} \label{eqsl}[X_1,X_2]=-X_3,\;[X_2,X_3]=X_1\esp [X_3,X_1]=X_2. \end{equation*}
The group of automorphisms of $\mathrm{sl}(2,\R)$ is generated by the automorphisms with the matrix given in the basis $(X_1,X_2,X_3)$ 
\[\mathrm{Rot}_{xy}=\left(\begin{array}{ccc}\cos(a)&\sin(a)&0\\-\sin(a)& \cos(a)&0\\0&0&1 \end{array}   \right),\; \mathrm{Boost}_{xz}=\left(\begin{array}{ccc}\cosh(a)&0&\sinh(a)\\0&1 &0\\
	\sinh(a)&0&\cosh(a) \end{array}\right),\mathrm{Boost}_{yz}=\left(\begin{array}{ccc}1&0&0\\0&\cosh(a) &\sinh(a)\\
	0&\sinh(a)&\cosh(a) \end{array}   \right).    \]
\item If $\G$ is a Lie algebra and $F$ is an automorphism of $\G$ we denote by 
$T^*F$ the automorphism of $T^*\G$ given by $T^*F=(F,(F^{-1})^*)$. This automorphism preserves the quadratic metric of $T^*\G$.

\end{enumerate}

\begin{Le}\label{le1} \begin{enumerate}
		\item Let $X$ be a non null vector in $\mathfrak{su}(2)$. Then there exists an automorphism of $\mathfrak{su}(2)$ which maps $X$ to $\al e_1$.
		\item Let $X$ be a non null vector in $\mathrm{sl}(2,\R)$ which is non colinear to $e_3$. Then there exists an automorphism of $\mathrm{sl}(2,\R)$ which maps $X$ to $ae_1+be_2$ with $(a,b)\not=(0,0)$.
		
	\end{enumerate}
	
\end{Le}

\begin{proof} It is straightforward by using the automorphisms given above. 
	\end{proof}

\begin{Le}\label{le2} There is no 5-dimensional subalgebra of $T^*\mathfrak{su}(2)$ and if $\h$ is a 4-dimensional degenerate subalgebra of $T^*\mathfrak{su}(2)$ then there exists an automorphism $F$ of $\mathfrak{su}(2)$ such that $T^*F(\h)=\R e_1\oplus \mathfrak{su}(2)^*$.
	
\end{Le}

\begin{proof} Let $\h$ be a 5-dimensional Lie subalgebra of $T^*\mathfrak{su}(2)$. Then $\dim\h\cap \mathfrak{su}(2)\geq2$. If $\dim\h\cap \mathfrak{su}(2)=2$ then $\h\cap \mathfrak{su}(2)$ is a subalgebra of $\mathfrak{su}(2)$.  But $\mathfrak{su}(2)$ has no 2-dimensional Lie algebra (see \cite[Proposition 5.2]{bou}).   If $\dim\h\cap \mathfrak{su}(2)=3$ then $\mathfrak{su}(2)\subset\h$ and $\h^\perp\subset \mathfrak{su}(2)^\perp=\mathfrak{su}(2)$.  But $[\h,\h^\perp]\subset \h^\perp$ an hence $\h^\perp$ is an ideal in $\mathfrak{su}(2)$ which is impossible.

	   Let $\h$ be a degenerate   4-dimensional Lie subalgebra of $T^*\mathfrak{su}(2)$.  Then we have three possibilities $\dim\h\cap\mathfrak{su}(2)=1,2$ or 3.
	
	$\bullet$ If $\dim\h\cap\mathfrak{su}(2)=3$ then $\h^\perp\subset \mathfrak{su}(2)$ and hence $\h^\perp$ is an ideal of $\mathfrak{su}(2)$ which is impossible.
	
	$\bullet$ If $\dim\h\cap\mathfrak{su}(2)=2$
	then $\h\cap\mathfrak{su}(2)$ is a 2-dimensional subalgebra of $\mathfrak{su}(2)$. But $\mathfrak{su}(2)$ has no 2-dimensional Lie algebra (see \cite[Proposition 5.2]{bou})
	
	$\bullet$ If $\dim\h\cap\mathfrak{su}(2)=1$ then, according to Lemma \ref{le1}, we can transform $\h$ by $T^*F$ where $F$ is an automorphism of $\mathfrak{su}(2)$ and choose
	$e_1$ as a generator of $\h\cap\mathfrak{su}(2)$. On the other hand,
	\[ \dim(\h^\perp\cap\mathfrak{su}(2))=6-\dim(\h+\mathfrak{su}(2))=6-4-3+1=0. \]Thus
	$\h^\perp\cap\mathfrak{su}(2)=\{0\}$.  There exists $\al_1\not=0,\al_2\not=0$  and $y_1,y_2\in\mathfrak{su}(2)$ 
	such that 
	\[ \h^\perp=\mathrm{span}\left(y_1+\al_1,y_2+\al_2\right). \]
	We must have $\{\al_1,\al_2\}$  linearly independent otherwise, we can find a non null element of $\h^\perp$ in $\mathfrak{su}(2)$. We have $\al_1(e_1)=\al_2(e_1)=0$. So $\mathrm{span}(\al_1,\al_2)=\mathrm{span}(e_2^*,
	e_3^*)$ and hence there exists $x_1,x_2\in\mathfrak{su}(2)$ such that 
	\[ \h^\perp=\mathrm{span}\left(x_1+e_2^*,x_2+e_3^*\right). \]
	 Now the fact that $\ad_{e_1}$ leaves $\h^\perp$ invariant implies that
	\[\begin{cases}
		[e_1,x_1]+[e_1,e_2^*]=[e_1,x_1]+e_3^*=\mu_1(x_1+e_2^*)+\nu_1(x_2+e_3^*),\\
		[e_1,x_2]+[e_1,e_3^*]=[e_1,x_2]-e_2^*=\mu_2(x_1+e_2^*)+\nu_2(x_2+e_3^*),\\
	\end{cases}  \]So
	\[ [e_1,x_1]=x_2\esp [e_1,x_2]=-x_1. \]
	This is equivalent to $x_1=a e_2-be_3$ and $x_2=b e_2+ae_3$ and hence
	\[ \h^\perp=\{a e_2-be_3+ e_2^*,b e_2+ae_3+ e_3^*\} \]
	The restriction of the metric to $\h^\perp$ is  degenerate if and only if $a=0$. So $\h=\mathrm{span}(e_1,e_1^*,-be_3+ e_2^*,b e_2+ e_3^*)$. But $[e_1^*,-be_3+ e_2^*]=be_2^*\in\h$ if and only if $b=0$. This completes the proof.
	\end{proof}

\begin{pr}\label{su}
	There is non $k$-symplectic structure on $T^*\mathfrak{su}(2)$.
\end{pr}

\begin{proof}
	According to Proposition \ref{nosym} and Lemma \ref{le2}, $T^*\mathfrak{su}(2)$ has no $k$-symplectic structure with $k=1$ or $k=5$.
	
	Let $(\h,D_1,D_2)$ be a 2-symplectic structure on 
	$T^*\mathfrak{su}(2)$ where $\h$ is a 4-dimensional subalgebra and $D_1,D_2$ two skew-symmetric derivations satisfying $D_i(\h)\subset\h^\perp$ and $\ker D_1\cap\ker D_2=\{0\}$. By virtue of Corollary \ref{co1} and Theorem \ref{h2}, $\h$ must be degenerate, and
	according to Lemma \ref{le2}, we can suppose that
	 $\h=\mathrm{span}(e_1,e_1^*,e_2^*,e_3^*)$ and $\h^\perp=\mathrm{span}(e_2^*,e_3^*)$. A skew-symmetric derivation $D$ of $T^*\mathfrak{su}(2)$ is inner (see Theorem \ref{h2}), i.e., $D=\ad_X$ where $X$ has coordinate $(x_1,\ldots,x_6)$ and we have
	 \[ D(e_1^*)=x_3e_2^*-x_2e_3^*,\;D(e_2^*)=-x_3e_1^*+x_1e_3^*\esp D(e_3^*)=x_2e_1^*+x_1e_2^*. \]
	 So $D(e_i^*)\subset\h^\perp$ for $i=1,\ldots,3$ if and only if $D(e_1^*)=0$. Thus $\ker D_1\cap\ker D_2\not=\{0\}$ which completes the proof.
	\end{proof}

\begin{Le}\label{le3}\begin{enumerate}
		\item Let $\h$ be a  5-dimensional degenerate subalgebra of $T^*\mathrm{sl}(\R,2)$ then there exists an automorphism $F$ of $\mathrm{sl}(\R,2)$ such that $T^*F(\h)=\mathrm{span}(e_1,e_3)\oplus \mathrm{sl}(\R,2)^*$.
		\item If $\h$ is a 4-dimensional degenerate subalgebra of $T^*\mathrm{sl}(\R,2)$ then there exists an automorphism $F$ of $\mathrm{sl}(\R,2)$ such that  $T^*F(\h)$ has one of the following forms
		\[\R e_1\oplus \mathrm{sl}(\R,2)^*,\;\R e_3\oplus \mathrm{sl}(\R,2)^*,\; \mathrm{span}\{e_1,e_3,e_1^*,e_2^*    \} 
		\quad\mbox{or}\quad \mathrm{span}\{e_1,e_3,e_2^*,e_3^*    \}.\]
		
	\end{enumerate}

	\end{Le}

\begin{proof}\begin{enumerate}\item Let $\h$ be a degenerate 5-dimensional Lie subalgebra of $T^*\mathrm{sl}(\R,2)$. Then $\dim\h\cap \mathrm{sl}(\R,2)\geq2$.    
		
		If $\dim\h\cap \mathrm{sl}(\R,2)=3$ then $\mathrm{sl}(\R,2)\subset\h$ and $\h^\perp\subset \mathrm{sl}(\R,2)^\perp=\mathrm{sl}(\R,2)$.  But $[\h,\h^\perp]\subset \h^\perp$ an hence $\h^\perp$ is an ideal in $\mathrm{sl}(\R,2)$ which is impossible.
	
If $\dim\h\cap \mathrm{sl}(\R,2)=2$ then $\h\cap \mathrm{sl}(\R,2)$ is a subalgebra of $\mathrm{sl}(\R,2)$.  According to  \cite[Proposition 5.2]{bou}), up to an automorphism, $\h\cap \mathrm{sl}(\R,2)=\mathrm{span}(e_1,e_3)$. So $\h^\perp+\mathrm{sl}(\R,2)=\mathrm{sl}(\R,2)\oplus\R e_2^*$. On the other hand, $\dim\h^\perp\cap \mathrm{sl}(\R,2)=0$ and hence $\h^\perp =\R(x+\al)$ with $\al\not=0$. We have $\al(e_1)=\al(e_3)=0$ so we can choose $\al=e_2^*$. Since $\h^\perp\subset\h$ ($\h$ being degenerate), we must have also $\al(x)=0$. Thus $x=ae_1+be_3$. But $[\h,\h^\perp]\subset\h^\perp$ so
 $[e_1,x]=mx\esp [e_3,x]=nx$ and hence $x=me_1$. Thus
 \[ \h=\mathrm{span}(e_1,e_3,-me_2+e_1^*,e_2^*,e_3^*). \]But $[e_3^*,-me_2+e_1^*]=me_1^*$ and $me_1^*\in\h$ if and only if $m=0$. Which complete the first part of the proof.

\item	     Let $\h$ be a    4-dimensional Lie degenerate subalgebra of $T^*\mathrm{sl}(\R,2)$.   We have three possibilities $\dim\h\cap\mathrm{sl}(\R,2)=1,2$ or 3. The same argument as in the proof of Lemma \ref{le2} shows that the case $\dim\h\cap\mathrm{sl}(\R,2)=3$ is impossible.
	 
	 $\bullet$ Suppose that  $\G_1=\h\cap\mathrm{sl}(\R,2)$ is a two dimensional subalgebra of $\mathrm{sl}(\R,2)$. According to \cite[Proposition 5.1]{bou},  we can suppose that $\G_1=\mathrm{span}(e_1,e_3)$. On the other hand, $\dim\h^\perp\cap\mathrm{sl}(\R,2)=6-\dim(\h+\dim\mathrm{sl}(\R,2))=1$. Denote by $\al$ a generator of $\h^\perp\cap\mathrm{sl}(\R,2)$. Then $\h\subset\al^\perp=\al^0\oplus \mathrm{sl}(\R,2)^*$ where $\al^0$ is the annihilator of $\al$. This implies that $\G_2=\h\cap \mathrm{sl}(\R,2)^*$ is 2-dimensional. So finally, $\h=\G_1\oplus\G_2$ and $\G_2$ is invariant by $\ad_{e_1}^*$ and $\ad_{e_3}^*$. In the dual basis, we have
	 \[ \ad_{e_1}^*=\left[\begin{array}{ccc}
	 	0 & 0 & 0 
	 	\\
	 	0 & 0 & 1 
	 	\\
	 	-2 & 0 & 0 
	 \end{array}\right]\esp 
 \ad_{e_3}^*=\left[\begin{array}{ccc}
 	2 & 0 & 0 
 	\\
 	0 & -2 & 
 	\\
 	0 & 0 & 0 
 \end{array}\right].
	  \]We have $\ad_{e_1}^*$  leaves invariant $\mathrm{span}(e_2^*,e_3^*)$ and $\G_2$ so it leaves invariant their intersection and since 0 is the only real eigenvalue of $\ad_{e_1}^*$, we get that $e_2^*\in\G_2$. We have also that
	   $\ad_{e_3}^*$  leaves invariant $\mathrm{span}(e_1^*,e_3^*)$ and $\G_2$ so it leaves invariant their intersection and hence either $e_1^*\in\G_2$ or $e_3^*\in\G_2$. Finally,
	  \[ \h=\mathrm{span}\{e_1,e_3,e_1^*,e_2^*    \} 
	  \quad\mbox{or}\quad \h=\mathrm{span}\{e_1,e_3,e_2^*,e_3^*    \}.\]

$\bullet$ $\dim\h\cap\mathrm{sl}(\R,2)=1$. According to Lemma \ref{le1}, we can suppose that a generator of 	 $\h\cap\mathrm{sl}(\R,2)$ is either $ae_1+be_2$ or $e_3$ with $(a,b)\not=0$.

Suppose that $a\not=0$ and hence we can choose  $e_1+be_2$ as a generator of $\h\cap\mathrm{sl}(\R,2)$. We must have $\h^\perp\cap\mathfrak{su}(2)=\{0\}$ and hence there exists $\al_1\not=0,\al_2\not=0$  and $x_1,x_2\in\mathfrak{su}(2)$ 
such that 
\[ \h^\perp=\mathrm{span}\left(x_1+\al_1,x_2+\al_2\right). \]
We must have $\{\al_1,\al_2\}$  linearly independent. We have $\al_1(e_1+be_2)=\al_2(e_1+be_2)=0$ and hence we can choose  $\al_1=-be_1^*+e_2^*$ and $\al_2=e_3^*$. Now the fact $\ad_{e_1}$ leaves $\h^\perp$ invariant implies that
\[\begin{cases}
	[e_1,x_1]+[e_1,e_2^*]-b[e_1,e_1^*]=[e_1,x_1]-2be_3^*=
	\mu_1(x_1+e_2^*-be_1^*)+\nu_1(x_2+e_3^*),\\
	[e_1,x_2]+[e_1,e_3^*]=[e_1,x_2]-e_2^*=\mu_2(x_1+e_2^*-be_1^*)+\nu_2(x_2+e_3^*),\\
\end{cases}  \]So $b=0$ and 
\[ [e_1,x_1]=0\esp [e_1,x_2]=-x_1. \]
This is equivalent to $x_1=\al e_1$ and $x_2=\frac{\al}2 e_3+\be e_1$ and hence
\[ \h^\perp=\{\al e_1+ e_2^*,\be e_1+\frac{\al}2 e_3+ e_3^*\} \]
Then 
\[ \h=\mathrm{span}\left(e_1,e_2^*,ae_2-e_1^*,\frac{a}2e_3-e_3^*  \right). \]
An one can see that $\h$ is a subalgebra if and only if $a=0$. Hence $\h=\R e_1\oplus\mathrm{sl}(2,\R)^*$ and $\h^\perp=\mathrm{span}(e_2^*,e_3^*)$.

Suppose that $e_3$ is a generator of $\h\cap\mathrm{sl}(\R,2)$. Then   $\h^\perp\cap\mathrm{sl}(\R,2)=\{0\}$ and  there exists $\al_1\not=0,\al_2\not=0$  and $y_1,y_2\in\mathfrak{su}(2)$ 
such that 
\[ \h^\perp=\mathrm{span}\left(x_1+\al_1,x_2+\al_2\right). \]
We must have $\{\al_1,\al_2\}$  linearly independent. We have $\al_1(e_3)=\al_2(e_3)=0$ and hence we can choose  $\al_1=e_1^*$ and $\al_2=e_2^*$. Now the fact $\ad_{e_3}$ leaves $\h^\perp$ invariant implies that
\[\begin{cases}
	[e_3,x_1]+[e_3,e_1^*]=[e_1,x_1]-2e_1^*=\mu_1(x_1+e_1^*)+\nu_1(x_2+e_2^*),\\
	[e_3,x_2]+[e_3,e_2^*]=[e_1,x_2]+2e_2^*=\mu_2(x_1+e_1^*)+\nu_2(x_2+e_2^*),\\
\end{cases}  \]So
\[ [e_3,x_1]=-2x_1\esp [e_3,x_2]=2x_2. \]
So $x_1=\al e_2$ and $x_2=\be e_1$. So
\[ \h^\perp=\{\al e_2+ e_1^*,\be e_1+ e_2^*\} \]
 Hence $\h=\R e_3\oplus\mathrm{sl}(2,\R)^*$ and $\h^\perp=\mathrm{span}(e_2^*,e_3^*)$.\qedhere

\end{enumerate}

\end{proof}

\begin{pr}
	There is non $k$-symplectic structure on $T^*\mathrm{sl}(\R,2)$.
\end{pr}

\begin{proof} By virtue of Proposition \ref{nosym}, $T^*\mathrm{sl}(\R,2)$ has no 1-symplectic structure.
	
Let $\h$ be a degenerate 5-dimensional Lie subalgebra of 	$T^*\mathrm{sl}(\R,2)$ and $D$ a skew-symmetric derivation of $T^*\mathrm{sl}(\R,2)$ such that $D(\h)\subset\h^\perp$. According to Lemma \ref{le3}, $\h=\spa(e_2,e_3)\oplus\sln^*$ and $\h^\perp=\R e_1^*$. Moreover, according to Theorem \ref{h2}, $D=\ad_{x+\al}$. So we must have $\ad_x^*(\sln^*)\subset \R e_1^*$ this implies that $\ad_x$ in restriction to $\mathrm{sl}(\R,2)$ has rank $\leq1$ which implies that $\ad_x=0$. So $D$ vanishes in restriction to  $\sln^*$. This shows that $T^*\mathrm{sl}(\R,2)$ has no 5-symplectic structure.

Let $\h$ be a nondegenerate    4-dimensional Lie subalgebra of $T^*\mathrm{sl}(\R,2)$. According to Lemma \ref{le3}, we have four cases:

\begin{enumerate}
	\item $\h=\R e_1\oplus \mathrm{sl}(\R,2)^*$. Let $D=\ad_X$ be a skew-symmetric derivation of $T^*\mathrm{sl}(2,\R)$ with $X=[x_1,\ldots,x_6]$ then
	$D(e_1^*),D(e_3^*)\in\h^\perp$ if and only if $D(e_2^*)=0)$.

	\item $\h= \R e_3\oplus \mathrm{sl}(\R,2)^*$. Let $D=\ad_X$ be a skew-symmetric derivation of $T^*\mathrm{sl}(2,\R)$ with $X=[x_1,\ldots,x_6]$ then $D(e_1^*)\in\h^\perp$ if and only if $D(e_1^*)=0$
	
	\item $\h = \mathrm{span}\{e_1,e_3,e_1^*,e_2^*    \}$ or 
	 $\mathrm{span}\{e_1,e_3,e_2^*,e_3^*    \}$. Let $D$ be a skew-symmetric derivation of $T^*\mathrm{sl}(\R,2)$. Then $D=\ad_X$ where $X=(x_1,\ldots,x_6)$ and
	 \[ D(e_2)=\left[\begin{array}{cccccc}
	 	0 & -2 x_{3} & x_{1} & -x_{6} & 0 & 2 x_{5} 
	 \end{array}\right]
	 \esp D(e_1^*)=\left[\begin{array}{cccccc}
	 	0 & 0 & 0 & -2 x_{3} & 0 & 2 x_{1} 
	 \end{array}\right].
	 \]If $\h=\mathrm{span}\{e_1,e_3,e_1^*,e_2^*    \}$ then  $\h^\perp=\mathrm{span}(e_3,e_2^*)$ and $D(e_2)\in\h^\perp$ if and only if $D(e_2)=0$.
	 If $\h=\mathrm{span}\{e_1,e_3,e_2^*,e_3^*    \}$ then  $\h^\perp=\mathrm{span}(e_1,e_2^*)$ and $D(e_1^*)\in\h^\perp$ if and only if $D(e_1^*)=0$. 
\end{enumerate}
	This shows that $\h$ cannot be the subalgebra of a 2-symplectic structure and hence $T^*\sln$ has no 2-symplectic structure.
	\end{proof}

{\renewcommand*{\arraystretch}{1.4}
	\begin{center}	
		\begin{tabular}{|c|c|c|}
			\hline
			Lie algebra&The non vanishing bracket&The metric\\
			\hline
			$\mathfrak{su}(2)$&
			$[e_1,e_2]=e_3,[e_2,e_3]=e_1,[e_3,e_1]=e_2.$&$c\mathrm{id}$\\
			\hline
			$\mathrm{sl}(2,\R)$&
			$[e_1,e_2]=e_3,[e_3,e_1]=2e_1,[e_3,e_2]=-2e_2.$&$c\{(1,2)=1,(3,3)=2\}$\\
			\hline
		\end{tabular}
		\captionof{table}{Indecomposable quadratic Lie algebras of dimension 3\label{1}}
\end{center}}

{\renewcommand*{\arraystretch}{1.4}
	\begin{center}	
		\begin{tabular}{|c|c|c|}
			\hline
			Lie algebra&The non vanishing bracket&The metric\\
			\hline
			$\mathrm{os}(4,\la)$&
			$[e_2,e_3]=e_1,[e_4,e_2]=\la e_3,[e_4,e_3]=-\la e_2$&$\{(2,2)=(3,3)=\frac1{\la},(1,4)=1\}$\\
			\hline
			$\G_{1,4}$&$[e_4,e_2]=e_2,\;[e_4,e_3]=-e_3,\;[e_2,e_3]=e_1$&$\{(1,4)=(2,3)=1\}$\\
			\hline
		\end{tabular}
		\captionof{table}{Indecomposable non abelian quadratic Lie algebras of dimension 4\label{2}}
\end{center}}

{\renewcommand*{\arraystretch}{1.4}
	\begin{center}	
		\begin{tabular}{|c|c|c|}
			\hline
			Lie algebra&The non vanishing bracket&The metric\\
			\hline
			$\G_{1,5}$&$[e_2,e_3]=e_1,[e_3,e_4]=-e_1,[e_5,e_2]=e_3,$&$\{(1,5)=-(2,2)=(3,3)=(4,4)=1\}$\\
			&$[e_5,e_3]=e_2-e_4,[e_5,e_4]=e_3$&\\
			\hline
		\end{tabular}
		\captionof{table}{Indecomposable non abelian quadratic Lie algebras of dimension 5\label{3}}
\end{center}}

{\renewcommand*{\arraystretch}{1.4}
	\begin{center}	
		\begin{tabular}{|c|c|c|}
			\hline
			Lie algebra&The non vanishing bracket&The metric\\
			\hline
			$\mathrm{osc}(6,(\la_1,\la_2))$&
			$[e_2,e_4]=[e_3,e_5]=e_1,[e_6,e_2]=\la_1 e_4$&$\{(2,2)=(3,3)=(4,4)=(5,5)=\frac1{\la} $\\
			&$[e_6,e_4]=-\la_1e_2,[e_6,e_3]=\la_2e_5,[e_6,e_5]=-\la_2e_3$&$(1,6)=1\}$\\		
			\hline
			$\mathfrak{l}_{2,\la}$&
			$[e_2,e_3]=e_1,[e_4,e_5]=\lambda e_1,[e_6,e_2]=e_3,$&$\{-(2,2)=(3,3)=(4,4)=(5,5)=1 $\\
			&$[e_6,e_3]=e_2,[e_6,e_4]=\la e_5,[e_6,e_5]=-\la e_4$&$(1,6)=1\}$\\
			\hline
			$\mathrm{so}(3,1)$&$[e_1,e_2]=e_3,[e_1,e_3]=-e_2,[e_1,e_4]=e_5,[e_1,e_5]
			=-e_4,$&$\{(1,1)=(2,2)=(3,3)=1$\\
			&$[e_2,e_3]=e_1,[e_2,e_4]=e_6,[e_2,e_6]=-e_4,[e_3,e_5]=
			e_6,$&$(4,4)=(5,5)=(6,6)=-1 \}$\\
			&$[e_3,e_6]=-e_5,[e_4,e_5]=-e_1,[e_4,e_6]=-e_2,
			[e_5,e_6]=-e_3.$&\\
			\hline
			$T^*\mathrm{sl}(2,\R)$&$[e_1,e_2]=e_3,[e_3,e_1]=2e_1,[e_3,e_2]=-2e_2$&$\{(1,4)=(2,5)=(3,6)=1 \}$\\
			&$[e_1,e_1^*]=2e_3^*,[e_1,e_3^*]=-e_2^*,[e_2,e_2^*]=-2e_3^*  $&\\
			&$[e_2,e_3^*]=e_1^*,[e_3,e_1^*]=-2e_1^*,[e_3,e_2^*]=2e_2^*$&\\
			\hline
			$T^*\mathfrak{su}(2)$&$[e_1,e_2]=e_3,[e_2,e_3]=e_1,[e_3,e_1]=e_2$&$\{(1,4)=(2,5)=(3,6)=1 \}$\\
			&$[e_1, e_2^*] = e_3^*, [e_1, e_3^*] = -e_2^*, [e_2, e_1^*] = -e_3^*,$&\\
			&$[e_2, e_3^*] = e_1^*, [e_3, e_1^*] = e_2^*, [e_3, e_2^*] = -e_1^*.$&\\
			\hline	
			$\mathfrak{n}_1(2,2)$&$[e_6,e_3]=e_2,[e_6,e_5]=e_4,[e_3,e_5]=e_1.$&$\{(1,6)=(2,5)=-(3,4)=1\}$\\
			\hline
			$\mathfrak{n}_2(2,2)$&
			$[e_6,e_2]=e_2+te_3,[e_6,e_3]=-te_2+e_3,[e_6,e_4]=-e_4+te_5$&$\{(1,6)=(2,5)=-(3,4)=1\}$\\
			&$[e_6,e_5]=-te_4-e_3, [e_2,e_4]=-te_1,[e_2,e_5]=e_1,$&\\	
			&$[e_3,e_4]=-e_1,[e_3,e_5]=-te_1. t>0$&\\
			\hline
			$\mathfrak{n}_3(2,2)$&$[e_6,e_2]=e_3,[e_6,e_3]=-e_2,[e_6,e_4]=\epsilon e_2+e_5,\e^2=1,$&$	\{(1,6)=(2,5)=-(3,4)=1\}$\\
			&$[e_6,e_5]=\epsilon e_3-e_4, [e_2,e_4]=-e_1,[e_3,e_5]=-e_1,[e_4,e_5]=\e e_1.$&\\	
			\hline
			$\mathfrak{n}_4(2,2)$&$	[e_6,e_2]=e_3,[e_6,e_3]=-e_2,[e_6,e_4]=te_5,$&$\{-(2,2)=-(3,3)=(4,4)=(5,5)=1  $\\	
			&$[e_6,e_5]=-te_4,  [e_2,e_3]=-e_1,[e_4,e_5]=te_1.\quad (t>0)$&
			$(1,6)=1\}$\\
			\hline
			$\mathfrak{n}_5(2,2)$&$[e_6,e_2]=e_2,[e_6,e_3]=e_2+e_3,[e_6,e_4]=-e_4,[e_6,e_5]=e_4-e_5,$&$\{(1,6)=(2,5)=-(3,4)=1\}$\\
			&$[e_2,e_5]=e_1,[e_3,e_4]=-e_1,[e_3,e_5]=e_1.$&\\
			\hline
			$\mathfrak{n}_6(2,2)$&$[e_6,e_2]=e_2,[e_6,e_3]=-e_3,[e_6,e_4]=te_4,$&$\{(1,6)=(2,3)=(4,5)=1  \}$\\
			&$[e_6,e_5]=-te_5,  [e_2,e_3]=e_1,[e_4,e_5]=te_1,t\geq1$&\\
			\hline
		\end{tabular}
		\captionof{table}{Indecomposable non abelian quadratic Lie algebras of dimension 6\label{4}}
\end{center}}

\end{document}